\documentclass[journal,draftcls,onecolumn]{IEEEtran}
\usepackage{epsfig,graphicx,amssymb,amsmath}
\usepackage{bm}
\usepackage{amsthm}
\usepackage{amsmath,bm}
\usepackage{subfigure}
\usepackage{algorithm, algorithmic}
\usepackage{amsfonts}
\usepackage[noadjust]{cite}  
\usepackage{url}
\usepackage{array,tabularx}
\usepackage{indentfirst}
\usepackage{color}
\usepackage{stfloats}

\ifCLASSINFOpdf
\else
\fi
\begin{document}
%
\title{Distributed Event-Triggered Algorithm for Convex Optimization with Coupled Constraints}
%


\author{Yi~Huang, Xianlin~Zeng, Ziyang Meng, \IEEEmembership{Senior Member, IEEE}, and Jian~Sun, 	\IEEEmembership{Senior Member, IEEE} 
\thanks{This work has been supported in part by the National
Natural Science Foundation of China under Grants 62103223, 62073035, 61833009 and U19B2029.

Yi Huang is with the School of Automation, Beijing Institute of Technology, Beijing 100081, China (e-mail: yihuang@bit.edu.cn).

Xianlin Zeng is with the State Key Laboratory of Intelligent Control and Decision of Complex Systems, School of Automation, Beijing
Institute of Technology, 100081, Beijing, China (xianlin.zeng@bit.edu.cn).

Ziyang Meng is with the Department of Precision
Instrument, Tsinghua University, Beijing 100084, China (e-mail: ziyangmeng@mail.tsinghua.edu.cn).

Jian Sun is  with the State Key Laboratory of Intelligent Control and Decision of Complex Systems and the School of Automation, Beijing Institute of Technology, Beijing 100081, China, and also with the Chongqing Innovation Center, Beijing Institute of Technology, Chongqing 401120, China (e-mail: sunjian@bit.edu.cn).

}}
\maketitle

\begin{abstract}
This paper develops a distributed primal-dual algorithm via event-triggered mechanism to solve a class of convex optimization problems subject to local set constraints, coupled equality and inequality constraints. Different from some existing distributed algorithms with the diminishing step-sizes, our algorithm uses the constant step-sizes, and is shown to achieve an exact convergence to an  optimal solution with an $O(1/k)$ convergence rate for general convex objective functions, where $k>0$ is the iteration number. Moreover, by applying event-triggered communication mechanism, the proposed algorithm can effectively reduce the communication cost without sacrificing the convergence rate. Finally, a numerical example is presented to verify the effectiveness of the proposed algorithm.
\end{abstract}


\newtheorem{Assumption}{Assumption}
 \newtheorem{Remark}{Remark}
 \newtheorem{Lemma}{Lemma}
 \newtheorem{Definition}{Definition}
  \newtheorem{Proposition}{Proposition}
 \newtheorem{Theorem}{Theorem}
 \newtheorem{Property}{Property}
 \newtheorem{Corollary}{Corollary}
 \newtheorem{Example}{Example}

\section{Introduction}
Distributed optimization has attracted considerable attention due to its wide applications in machine learning, multi-robot localization and sensor networks and \cite{a2,a3,a4,a5}. Many works (e.g., \cite{a6,a8,a9,a10,a11,a12}) considered the optimal consensus problem, in which the global objective function is a sum of local objective ones with a common variable. In contrast to the optimal consensus problem, we consider an another distributed optimization problem with coupled constraints, in which each agent has its own objective function and local decision variable while all the agents' local decision variables are coupled with the global equality or inequality constraint. This optimization problem arises in some practical applications (e.g. economic dispatch and flow control in smart grids \cite{1,2,3,4,3a}).

To deal with the coupled constraints, many distributed algorithms have been proposed in \cite{1,2,3,3a,4,5,6,7,8,8a,9,10}. Note that affine coupled equality or inequality constraints were considered in \cite{1,2,3,3a,4,5}, and it is more challenging that the coupled inequality constraint has a nonaffine structure. In \cite{6}, a consensus-based primal-dual perturbation algorithm was proposed to solve the optimization problem with coupled nonlinear inequality constraints. The relaxation and duality-based distributed algorithm of \cite{7}, and dual decomposition distributed algorithms of \cite{8,8a} were developed to deal with the coupled nonlinear inequality constraints. Note that \cite{8a} only converges to a suboptimal solution, and \cite{6,7,8} use the diminishing step-sizes with only asymptotic convergence. The local set constraints, coupled affine equality and nonlinear inequality constraints are both considered in \cite{9} and \cite{10}, and two distributed optimization algorithms were developed. \cite{9} only proved asymptotic convergence and the analysis of the convergence rate was not given, and \cite{10} was shown to achieve an $O(\frac{\text{ln} k}{\sqrt{k}})$ convergence rate with the step-size $1/\sqrt{k+1}$.

Note that most distributed algorithms mentioned above use the periodic communication mechanism, i.e., each agent needs to communicate with its neighbors at each sampling time or iteration instant. If the sampling time or iteration step-size is small, the algorithms via periodic communication lead to high communication cost. In contrast with periodic communication one, event-triggered mechanism is a more communication-effective approach, in which each agent only communicates with the neighbors at the event-triggering times determined by some triggering rules. Based on this attractive property,  \cite{11,12,13} proposed some event-triggered distributed optimization algorithms. Event though the communication cost can be reduced in \cite{11,12,13}, the convergence rates are compromised since a diminishing step-size is used. To accelerate the convergence rate, the authors of \cite{14} and \cite{15} proposed two event-based distributed optimization algorithms with constant step-sizes, in which \cite{14} only guarantees the convergence to a suboptimal solution, and \cite{15} achieves an exact convergence but cannot be applied directly to the coupled nonlinear constrained problem.

Inspired by the above discussions, the objective of this paper is to design a distributed event-triggered algorithm with constant step-sizes to solve a class of convex optimization problems with local set constraints, coupled affine equality and nonlinear inequality constraints. Compared with the related results, the main contributions of this paper are three-fold:
%

\begin{enumerate}
  \item [\textbf{c1)}] Based on the primal-dual method and event-triggered mechanism, we develop a novel distributed algorithm with constant step-sizes to solve the optimization problem with local set constraints, coupled equality and inequality constraints, and provide the explicit selection criteria of constant step-sizes.
  \item [\textbf{c2)}] The proposed algorithm is shown to achieve an exact convergence to an optimal solution with an $O(1/k)$ convergence rate for general convex objective functions. The convergence results of our algorithm outperform those of \cite{6,7,8,8a,9,10}. In particular, a suboptimal solution is obtained in \cite{8a}, and only asymptotic convergence is achieved in \cite{6,7,8} and \cite{9}, while the convergence rate of $O(\frac{\text{ln} k}{\sqrt{k}})$ in \cite{10} is slower than $O(1/k)$ of our proposed algorithm.
  \item [\textbf{c3)}] Compared with the distributed algorithms in \cite{6,7,8,8a,9,10} with periodic communication, the proposed algorithm via event-triggered mechanism can effectively reduce the communication cost. Moreover, in contrast to some event-triggered algorithms via the diminishing step-sizes in \cite{11,12,13} where the convergence rates are compromised, the proposed event-triggered algorithm via the constant step-sizes does not sacrifice the convergence rate.
\end{enumerate}

The rest of this paper is organized as follows. Section II provides some preliminaries and formulates the problem. Section III presents a distributed event-triggered optimization algorithm and its convergence analysis. Section IV gives a simulation example and Section V draws some conclusions.

\section{Preliminaries and formulation}
Notation: Let $\mathbb{R}_{\ge 0}$, $\mathbb{R}_{\le 0}$ and $\mathbb{R}$ be the sets of nonnegative, nonpositive and real numbers. $\mathbb{N}$ denotes the set of natural numbers. Let $I_{n}$ be an $n\times n$ identity matrix and $\bm 1_{n}$ be a $n$-dimensional vector with all entries being 1. $\Vert \cdot\Vert$ denotes the Euclidean norm of a vector or a matrix. Let $\text{col}(x_{i})^{n}_{i=1}$ be a column stack of a vector or a matrix $x_{i}$ over $i$ from $1$ to $n$, $\otimes$ represents the Kronecker product and $\lambda_{\max}(A)$ denotes the largest eigenvalues of matrix $A$. $\textbf{ker}(M)$ and $\textbf{Im}(M)$ denote the kernel space and image space of matrix $M$, respectively.

\subsection{Graph Theory}
A weighted undirected graph is described by $\mathcal{G}=(\mathcal{V},\mathcal{E}, A)$, where $\mathcal{V}=\{1,2,\ldots, N\}$ is the node set, and $\mathcal{E} \subset \mathcal{V}\times \mathcal{V}$ is the edge set, $A=[a_{ij}] \in \mathbb{R}^{N\times N}$ is an adjacency matrix with $a_{ij}>0$ if $(j, i) \in \mathcal{E}$, and $a_{ij}=0$ otherwise. $\mathcal{N}_{i}=\{j~|~ (j, i) \in \mathcal{E}\}$ denotes the neighbors of node $i\in \mathcal{V}$. $L=D-A \in \mathbb{R}^{N\times N}$ is the Laplacian matrix of $\mathcal{G}$, where $D=\text{diag}\{d_{1},d_{2}, \ldots, d_{N}\}$ and $d_{i}=\sum^{N}_{j=1}a_{ij}$. For an undirected graph $\mathcal{G}$, $L$ is a symmetric and positive semi-definite matrix, and its eigenvalues are ordered as $0\le \lambda_{2}(L)\le \ldots \le \lambda_{N}(L). $

\subsection{Convex property}
A differentiable function $f:\mathbb{R}^{m} \to \mathbb{R}$ is convex if $f(x_{1})-f(x_{2})\ge \nabla f^{T}(x_{2})(x_{1}-x_{2})$ for $\forall x_{1}, x_{2} \in \mathbb{R}^{m}$, where $\nabla f$ is the gradient of $f$. The gradient function $\nabla f$ is $\theta$-Lipschitz continuous with a positive constant $\theta$ if $\Vert \nabla f(x_{1})-\nabla f(x_{2})\Vert \le \theta \Vert x_{1}-x_{2}\Vert$ for $\forall x_{1}, x_{2}\in \mathbb{R}^{m}$. A set $\Omega$ is convex if for any $0<\mu<1$, $x, y\in \Omega$ implies that $\mu x+(1-\mu)y \in \Omega$. Given a closed and convex set $\Omega$, the projection of $x$ on $\Omega$ is denoted as $\mathcal{P}_{\Omega}(x)=\text{argmin}_{y\in \Omega} \Vert x-y\Vert$, which satisfies the property $(x-\mathcal{P}_{\Omega}(x))^{T}(y-\mathcal{P}_{\Omega}(x))\le 0, \forall x\in \mathbb{R}, y\in \Omega$. The norm cone of the set $\Omega$ at $x \in \Omega$ is defined as
\begin{align}\label{m1}
\mathcal{N}_{\Omega}(x)=\{w\in \mathbb{R}^{n}|~w^{T}(v-x)\le 0, \forall v\in \Omega \},
\end{align}
From Lemma 2.38 in \cite{15b}, it follows that $\mathcal{P}_{\Omega}(x+w)=x$ for any $w \in \mathcal{N}_{\Omega}(x)$.
%

\subsection{Problem formulation}
Consider a multi-agent network of $N$ agents over a graph $\mathcal{G}$, in which each agent $i\in \mathcal{V}$ privately has a local objective function $f_{i}(x_{i}): \mathbb{R}^{n}\to \mathbb{R}$ and local decision variable $x_{i}\in \mathbb{R}^{n_{i}}$. The objective is to cooperatively solve the following nonsmooth optimization problem
\begin{subequations}\label{1}
\begin{align}
&\min_{x\in \Omega} f(x)=\sum^{N}_{i=1}f_{i}(x_{i}) \\
&\text{s.t. }~~ g(x)=\sum^{N}_{i=1}g_{i}(x_{i})\le 0, h(x)=\sum^{N}_{i=1}h_{i}(x_{i})=0, \label{1.1}
\end{align}
\end{subequations}
where $x=\text{col}(x_{1},\ldots, x_{N})\in \mathbb{R}^{n}, n=\sum^{N}_{i=1}n_{i}$ is a collection of all the decision variables, $f(x): \mathbb{R}^{n}\to \mathbb{R}$ is the global objective function, $\Omega=\Omega_{1}\times \ldots \times \Omega_{N}$ is the Cartesinan product of local constraint $\Omega_{i}, i\in \mathcal{V}$, $g(x)=\sum^{N}_{i=1}g_{i}(x_{i})\le 0$ and $h(x)=\sum^{N}_{i=1}h_{i}(x_{i})=0$ are the coupled inequality and equality constraints, where $g_{i}(x_{i}): \mathbb{R}^{n_{i}} \to \mathbb{R}^{p}$ and $h_{i}(x_{i}): \mathbb{R}^{n_{i}} \to \mathbb{R}^{q}$ are the local constraint functions only accessible for each agent $i \in \mathcal{V}$. To proceed, some basic assumptions are given.

\begin{Assumption}
The graph $\mathcal{G}$ is undirected and connected.
\end{Assumption}

\begin{Assumption}
\begin{enumerate}
  \item [(i)] The local functions $f_{i}, g_{i}$ and $ h_{i}$ are all convex and differentiable on $\Omega_{i}$, and $\Omega_{i}$ is closed and convex for $\forall i\in \mathcal{V}$.
  \item [(ii)] $\nabla f_{i}(x_{i}), \nabla g_{i}(x_{i})$ and $g_{i}(x_{i})$ are Lipschitz continuous on $\Omega_{i}$, and $h_{i}(x_{i})$ is an affine function, i.e., $h_{i}(x_{i})=B_{i}x_{i}+b_{i}$ with $B_{i} \in \mathbb{R}^{q\times n_{i}}, b_{i} \in \mathbb{R}^{n_{i}}$.
  \item [(iii)] The Slater's condition is satisfied, i.e., there exists a point $\hat{x}\in \text{relint}(\Omega)$ such that $g(\hat{x})\le 0$ and $h(\hat{x})=0$, where $\text{relint}(\Omega)$ is the relative interior of  $\Omega$.
\end{enumerate}
\end{Assumption}
\begin{Remark}
The optimization problem \eqref{1} is formulated in a general form, which includes many existing ones in \cite{1,2,3,3a,4,5,6,7,8,8a} as a special one. In particular, if the coupled affine equality constraint is specified as $(L\otimes I)x=0$, the optimization problem \eqref{1} can be transformed into the optimal consensus problem. In addition, we require the local objective function $f_{i}$ to be just convex rather than strongly convex in \cite{2,3,4,3a}, and relax the compactness requirement of $\Omega_{i}$ in \cite{6,7,8}.
\end{Remark}
For notational simplicity, we define
\begin{align*}
\psi_{i}(x_{i})=\begin{bmatrix} g_{i}(x_{i})\\ h_{i}(x_{i})\end{bmatrix} \in \mathbb{R}^{m}, \Upsilon=\mathbb{R}^{p}_{\le 0}\times {0_{q}}, \Theta=\mathbb{R}^{p}_{\ge 0}\times \mathbb{R}^{q}
\end{align*}
with $m=p+q$. Then, the coupled equality and inequality constraints of \eqref{1} can be written as
\begin{align}\label{2}
\psi(x)=\sum^{N}_{i=1} \psi_{i}(x_{i}) \in \Upsilon.
\end{align}
For problem \eqref{1}, the Lagrangian function is obtained as $L(x, \lambda)=f(x)+\lambda^{T}\psi(x)=\sum^{N}_{i=1}f_{i}(x_{i})+\sum^{N}_{i=1}\lambda^{T}\psi_{i}(x_{i})$, where $\lambda \in \Theta$ is the dual variable or Lagrange multiplier with respect to the coupled constraint \eqref{2}. Note from Assumption 2 that the strong duality condition can be satisfied. Then, based on the Karush-Kuhn-Tucker (KKT) conditions, we can obtain the following lemma.

\begin{Lemma}
Suppose that Assumptions 1-2 hold. $x^{*}=\text{col}(x^{*}_{i})^{N}_{i=1} \in \mathbb{R}^{n}$ is an optimal solution to problem \eqref{1} if and only if there exists $\lambda^{*}_{0} \in \mathbb{R}^{m}$ such that
\begin{subequations}\label{2a}
\begin{align}
0 \in \nabla f(x^{*})+\nabla \psi^{T}(x^{*})\lambda^{*}_{0}+\mathcal{N}_{\Omega}(x^{*}), \label{2a.1}\\
\lambda^{*}_{0} \in \Theta, \psi(x^{*}) \in \Upsilon, (\lambda^{*}_{0})^{T}\psi(x^{*})=0.\label{2a.2}
\end{align}
\end{subequations}
where $\nabla f(x^{*})$ is the gradient of $f(x)$ at $x^{*}$, and $\nabla \psi(x^{*})=[\nabla \psi_{1}(x^{*}_{1}), \ldots, \nabla \psi_{N}(x^{*}_{N})]\in \mathbb{R}^{m\times n}$ is the Jacobian matrix of vector function $\psi(x)$ at $x^{*}$.
\end{Lemma}

\section{Main results}
In this section, we first develop a novel distributed event-triggered algorithm with constant step-sizes to solve the optimization problem \eqref{1}. Subsequently, the convergence analysis of the proposed algorithm is provided.

\subsection{Distributed algorithm design}
We develop the following event-triggered primal-dual distributed algorithm. Each agent $i\in \mathcal{V}$ owns three state variables $(x_{i,k}, \lambda_{i,k}, s_{i,k})\in \mathbb{R}^{n_{i}}\times \mathbb{R}^{m}\times \mathbb{R}^{m}$ at the iteration $k$, and their update rules are given as
\begin{subequations}\label{5}
\begin{align}
&x_{i,k+1}=\mathcal{P}_{\Omega_{i}} \Big (x_{i,k}-2\alpha \varpi_{i}(x_{i,k},\lambda_{i,k}) \notag \\
&\quad~~~~~+\alpha \varpi_{i}(x_{i,k-1}, \lambda_{i,k-1}) \Big), \label{5.1}\\
&\lambda_{i,k+1}=\mathcal{P}_{\Theta}\Big(\lambda_{i,k}+2\alpha \psi_{i}(x_{i,k})-\alpha \psi_{i}(x_{i,k-1}) \notag  \\
&\quad~~~~~-\alpha s_{i,k}-\alpha\beta \sum_{j\in \mathcal{N}_{i}} a_{ij}(\tilde{\lambda}_{i,k}-\tilde{\lambda}_{j,k})\Big), \label{5.2}\\
& s_{i,k+1}=s_{i,k}+\beta \sum_{j\in \mathcal{N}_{i}} a_{ij}(\tilde{\lambda}_{i,k+1}-\tilde{\lambda}_{j,k+1}), \label{5.3}
\end{align}
\end{subequations}
where $\varpi_{i}(x_{i,k},\lambda_{i,k})=\nabla f_{i}(x_{i,k})+\nabla \psi^{T}_{i}(x_{i,k})\lambda_{i,k}$, $\alpha$ and $\beta$ are both constant step-sizes that will be specified later. The update of \eqref{5} is developed by using the primal-dual method, in which $x_{i,k}$ is a primal variable, $\lambda_{i,k}$ is a dual variable and $s_{i,k}$ is an auxiliary variable. In addition, $\tilde{\lambda}_{i,k}$ denotes the information that agent $i$ broadcasts to its neighbors at the latest triggering time instant of the iteration $k$, which is defined as
\begin{align*}
\tilde{\lambda}_{i,k}=\begin{cases} \lambda_{i,k}, & k\in \mathcal{K}_{i} \\
\tilde{\lambda}_{i,k-1}, & \text{otherwise}, \end{cases}
\end{align*}
where $\mathcal{K}_{i}=\{k^{1}_{i}, k^{2}_{i}, k^{3}_{i}, \ldots \}, k^{j}_{i} \in \mathbb{N}$ is the set of the triggering times of agent $i$, and it is determined by the following triggering rule
\begin{align}\label{4}
& k^{l+1}_{i}=\min\{k\in \mathbb{N} | k>k^{l}_{i}, \Vert \lambda_{i,k}-\tilde{\lambda}_{i,k-1}\Vert>\varepsilon_{i,k}\},
\end{align}
where $\varepsilon_{i,k}\ge 0$ is the event-triggering threshold that will be given later. We set that the events of all the agents are automatically triggered at $k=0$, and then the set $\mathcal{K}_{i}$ can be well defined.

Some remarks of algorithm \eqref{5} are given as follows.
\begin{enumerate}
  \item [\textbf{i)}] The update of the primal variable $x_{i,k}$ is the ``gradient+momentum" type by using the optimistic gradient method in \cite{17a}, which includes the negative gradient term $-\varpi_{i}(x_{i,k},\lambda_{i,k})$ of the function $f_{i}(x_{i})+\lambda^{T}_{i}\psi_{i}(x_{i})$ and the negative momentum term $-(\varpi_{i}(x_{i,k},\lambda_{i,k})-\varpi_{i}(x_{i,k-1},\lambda_{i,k-1}))$.
  \item [\textbf{ii)}] The update of the dual variable $\lambda_{i,k}$ consists of the ``gradient+momentum" component $2\alpha \psi_{i}(x_{i,k})-\alpha \psi_{i}(x_{i,k-1})$ and event-triggered communication, which aims to guarantee the satisfaction of the coupled constraints \eqref{1.1}.
  \item [\textbf{iii)}] The update of the auxiliary $s_{i,k}$ is to achieve consensus of $\lambda_{i}$ for $\forall i\in \mathcal{V}$, which uses the latest event-triggered information $\sum_{j\in \mathcal{N}_{i}} a_{ij}(\tilde{\lambda}_{i,k+1}-\tilde{\lambda}_{j,k+1})$.
\end{enumerate}

From the triggering rule \eqref{4}, we know that if the derivation between the state $\lambda_{i,k}$ at the current iteration $k$ and the state $\tilde{\lambda}_{i,k-1}$ at the latest event-triggering time instant of the iteration time $k-1$ exceeds the threshold value $\varepsilon_{i,k}$, then a new event is triggered, i.e., $\tilde{\lambda}_{i,k}=\lambda_{i,k}$ and agent $i$ will broadcast $\tilde{\lambda}_{i,k}$ to its neighbors. Otherwise, $\tilde{\lambda}_{i,k}=\tilde{\lambda}_{i,k-1}$ and agent $i$ does not broadcast any information at the iteration $k$.

It follows from \eqref{5} that each agent $i\in \mathcal{V}$ only needs the state information $\tilde{\lambda}_{j,k}$ that is broadcasted from its neighbors $j\in \mathcal{N}_{i}$ at the nearest triggering instant of $k$. It is obvious that the proposed algorithm can be implemented in a distributed manner. The main steps of implementing algorithm \eqref{5} are shown in the following Algorithm 1.

\begin{algorithm}[htb]
	\caption{Event-triggered distributed optimization algorithm} \label{alg2}
\begin{algorithmic}[1]
\STATE \textbf{Initiation:} \\
   $x_{i,0} \in \mathbb{R}^{n_{i}}, \lambda_{i,0} \in \mathbb{R}^{m}, s_{i,0}=0$ for $\forall i\in \mathcal{V}$, \\
   Set $x_{i,-1}=x_{i,0}$ and $\lambda_{i,-1}=\lambda_{i,0}$.
   Broadcast $\lambda_{i,0}$ to its neighbors.\\
\STATE \textbf{for} each agent $i\in \mathcal{V}$ \textbf{do}\\
\STATE ~~Perform the updates \eqref{5.1} and \eqref{5.2};
\STATE ~~Test the event-triggered rule \eqref{4};
\STATE ~~\textbf{if} triggered  \textbf{then}
\STATE ~~~~Broadcast $\lambda_{i,k+1}$ to its neighbors;
\STATE ~~\textbf{end if}
\STATE ~~ Update the local update \eqref{5.3};
\STATE \textbf{end for}
\STATE Set $k=k+1$.
\end{algorithmic}
\end{algorithm}

Let $e_{i,k}=\tilde{\lambda}_{i,k}-\lambda_{i,k}$ be the error between the state $\lambda_{i,k}$ at the current iteration $k$ and the latest event-triggering time instant. For the case of $k \notin \mathcal{K}_{i}$, it follows that $\tilde{\lambda}_{i,k}=\tilde{\lambda}_{i,k-1}$. According to the event-triggering rule \eqref{4}, we obtain that if $\Vert e_{i,k}\Vert>\varepsilon_{i,k}$, it implies that $k\in \mathcal{K}_{i}$ and it occurs a contradiction. Then, we have that $\Vert e_{i,k}\Vert \le  \varepsilon_{i,k}$. On the other hand, if $k\in \mathcal{K}_{i}$, it implies that $e_{i,k}=0$. It then follows that $\Vert e_{i,k}\Vert \le \varepsilon_{i,k}$ holds for all $k\ge 0$. For the triggering threshold $\varepsilon_{i,k}$, we make the following assumption.
\begin{Assumption}
Denote $E_{k}=\max_{i\in \mathcal{V}}(\varepsilon_{i,k})$ for $\forall k \in \mathbb{N}$. $E_{k}$ is nonincreasing and summable, i.e., $\sum^{\infty}_{k=0}E_{k}<\infty$.
\end{Assumption}

Let $x_{k}=\text{col}(x_{i,k})^{N}_{i=1}, \lambda_{k}=\text{col}(\lambda_{i,k})^{N}_{i=1}, s_{k}=\text{col}(s_{i,k})^{N}_{i=1}, \nabla f(x_{k})=\text{col}(\nabla f_{i}(x_{i,k}))^{N}_{i=1}$, $\nabla \tilde{\psi}(x_{k})=\text{diag}($ $\nabla \psi_{i}(x_{i,k}))^{N}_{i=1}, \tilde{\psi}(x_{k})=\text{col}(\psi_{i}(x_{i,k}))^{N}_{i=1}, e_{k}=\text{col}(e_{i,k})^{N}_{i=1}$, $\mathcal{P}_{\Omega}(\cdot)=\text{col}(\mathcal{P}_{\Omega_{i}}(\cdot))^{N}_{i=1}$. According to the definition of errors $e_{i,k}$, one has that $\tilde{\lambda}_{k}=\text{col}(\tilde{\lambda}_{i,k})^{N}_{i=1}$ $=\lambda_{k}+e_{k}$. Then, algorithm \eqref{5} can be written in a compact form
\begin{align}\label{6}
\begin{cases}
x_{k+1}=\mathcal{P}_{\Omega}\big (x_{k}-2\alpha(\nabla f(x_{k})+\nabla \tilde{\psi}^{T}(x_{k})\lambda_{k})\\
~~~~~~~~~+\alpha(\nabla f(x_{k-1})+\nabla \tilde{\psi}^{T}(x_{k-1})\lambda_{k-1})\big )\\
\lambda_{k+1}=\mathcal{P}_{\bm \Theta}\big (\lambda_{k}+2\alpha \tilde{\psi}(x_{k})-\alpha \tilde{\psi}(x_{k-1})\\
~~~~~~~~~-\alpha s_{k}-\alpha\beta(L\otimes I_{m})(\lambda_{k}+e_{k})\big )\\
s_{k+1}=s_{k}+\beta (L\otimes I_{m})(\lambda_{k+1}+e_{k+1}),
\end{cases}
\end{align}
where $L\in \mathbb{R}^{N\times N}$ is the Laplacian matrix of graph $\mathcal{G}$, and $\mathcal{P}_{\bm \Theta}=\bm 1_{N}\otimes \mathcal{P}_{\Theta}$.

The following proposition reveals the relationship between the fixed point of algorithm \eqref{6} and the optimal solution of problem \eqref{1}.
\begin{Proposition}
Suppose that Assumptions 1-2 hold. $x^{*}$ is an optimal solution of problem \eqref{1} if and only if $(x^{*}, \lambda^{*}, s^{*})$ is a fixed point of \eqref{6}.
\end{Proposition}
\begin{proof}
\textbf{Sufficiency:} Let $(x^{*}, \lambda^{*}, s^{*})$ be a fixed point of algorithm \eqref{6}, and one has that
\begin{subequations}\label{7}
\begin{align}
&\mathcal{P}_{\Omega}(x^{*}-\alpha(\nabla f(x^{*})+\nabla \tilde{\psi}^{T}(x^{*})\lambda^{*}))=x^{*},\label{7.1}\\
&\mathcal{P}_{\bm \Theta}(\lambda^{*}+\alpha \tilde{\psi}(x^{*})-\alpha s^{*})=\lambda^{*}, \label{7.2}\\
&(L\otimes I_{m})\lambda^{*}=0.\label{7.3}
\end{align}
\end{subequations}
Under Assumption 1, we note from \eqref{7.3} that $\lambda^{*}=\bm 1_{N}\otimes \lambda^{*}_{0}$ for some vector $\lambda^{*}_{0}\in \mathbb{R}^{m}$. According to \eqref{7.1}, we have that $-\alpha(\nabla f(x^{*})+\nabla \tilde{\psi}^{T}(x^{*})\lambda^{*})\in \mathcal{N}_{\Omega}(x^{*})$. Based on $\lambda^{*}=\bm 1_{N}\otimes \lambda^{*}_{0}$, one can derive that $0 \in \nabla f(x^{*})+\nabla \psi^{T}(x^{*})\lambda^{*}_{0}+\mathcal{N}_{\Omega}(x^{*})$, which implies that \eqref{2a.1} holds. Under the initial value $s_{i,0}=0$, from the last equation of \eqref{6}, we have that $s_{t+1}=\beta L \sum^{t+1}_{k=1}\tilde{\lambda}_{k}$ for any $t\ge 0$. This implies that $s^{*} \in \textbf{Im}(L)$, and then there exists $y^{*} \in \mathbb{R}^{Nm}$ such that $s^{*}=(L\otimes I_{m})y^{*}$. In addition, from \eqref{7.2}, we obtain that
\begin{align}\label{8}
&\lambda^{*}\in \bm \Theta, (\lambda^{*})^{T}(\tilde{\psi}(x^{*})- s^{*})=0, \tilde{\psi}(x^{*})- s^{*} \in \bm \Upsilon,
\end{align}
where $\bm \Theta=\bm 1_{N}\otimes \Theta$ and $\bm \Upsilon= \bm 1_{N} \otimes \Upsilon$. Substituting $\lambda^{*}=\bm 1_{N}\otimes \lambda^{*}_{0}$ and $s^{*}=(L\otimes I_{m})y^{*}$ into \eqref{8}, we obtain that $\lambda^{*}_{0}\in \Theta$ and $(\lambda^{*}_{0})^{T}\psi(x^{*})=0$. Moreover, by left multiplying $\bm 1^{T}_{N}\otimes I_{m}$ with the last equation of \eqref{8}, one can derive that $(\bm 1^{T}_{N}\otimes I_{m})\tilde{\psi}(x^{*})=\sum^{N}_{i=1}\psi_{i}(x^{*}_{i})=\psi(x^{*})\in \Upsilon$. This implies that \eqref{2a.2} holds. According to Lemma 1, we have that $x^{*}$ is an optimal solution of problem \eqref{1}

\textbf{Necessity:} If $x^{*}$ is an optimal solution of problem \eqref{1}, from Lemma 1, one has that there exists $\lambda^{*}_{0}$ such that \eqref{2a} holds. By setting $\lambda^{*}=\bm 1_{N}\otimes \lambda^{*}_{0}$, we know that $(x^{*}, \lambda^{*})$ satisfies \eqref{7.1} and \eqref{7.3}. The rest is to prove the existence of $s^{*}$ satisfying \eqref{7.2}. According to \eqref{2a}, we have that $\lambda^{*} \in \bm \Theta$ and $\psi(x^{*})=(\bm 1^{T}_{N}\otimes I_{m})\tilde{\psi}(x^{*}) \in  \Upsilon$. From $\textbf{Ker}(\bm 1^{T}_{N})=\textbf{Im}(L)$, there exists a constant $s^{*}\in \textbf{Im}(L)$ such that $\tilde{\psi}(x^{*})-s^{*}\in \bm \Upsilon$. Moreover, we obtain that $(\lambda^{*})^{T}(\tilde{\psi}(x^{*})-s^{*})=0$ since $\psi(x^{*})=0$ and $\bm 1^{T}_{m}L=0$. Based on \eqref{8}, we know that \eqref{7.2} holds. As a result, we conclude that $(x^{*}, \lambda^{*}, s^{*})$ is a fixed point of \eqref{6}.
\end{proof}

%

\subsection{Convergence analysis}
We provide the convergence analysis of \eqref{6}. For convenience, we define the variables $\xi_{k}=[x^{T}_{k}, \lambda^{T}_{k}]^{T} \in \mathbb{R}^{r}, r=n+Nm$, $z_{k}=[0^{T}_{n},s^{T}_{k}]^{T}$, $\tilde{e}_{k}=[0^{T}_{n}, (e_{k})^{T}]^{T}$. Then, algorithm \eqref{6} can be written as
\begin{subequations}\label{9}
\begin{align}
\xi_{k+1}&=\mathcal{P}_{\Lambda}(\xi_{k}-2\alpha \Phi(\xi_{k})+\alpha \Phi(\xi_{k-1}) \label{9a}\\
&\quad-\alpha z_{i,k} -\alpha \beta L_{s}(\xi_{k}+\tilde{e}_{k})) \notag\\
z_{k+1}&=z_{k}+\beta L_{s}(\xi_{k+1}+\tilde{e}_{k+1}), \label{9b}
\end{align}
\end{subequations}
where $\Lambda=\Omega\times \bm \Theta$, $L_{s}=\text{diag}(0_{n},L\otimes I_{m})$ and $\Phi(\xi_{k})$ is
\begin{align}\label{10}
\Phi(\xi_{k})=\begin{bmatrix} \nabla f(x_{k})+\nabla \tilde{\psi}^{T}(x_{k})\lambda_{k} \\ -\tilde{\psi}(x_{k}) \end{bmatrix}.
\end{align}
Under Assumption 2.2, one has that $\Phi(\xi_{k})$ satisfies the Lipshitz continuity, i.e., $\Vert \Phi(\xi_{k+1})-\Phi(\xi_{k})\Vert\le \kappa_{c} \Vert \xi_{k+1}-\xi_{k}\Vert$, where $\kappa_{c}>0$ is a Lipschitz constant.

The convergence result of algorithm \eqref{5} is shown in the following theorem and its proof can be seen in Appendix A.
\begin{Theorem}
Suppose that Assumptions 1-3 hold and the constant step-sizes $\alpha, \beta$ satisfy that $\alpha <\frac{1}{3\kappa_{c}}$ and  $\beta\le \frac{1-3\alpha\kappa_{c}}{\alpha\lambda_{\max}(L)}$. The distributed algorithm \eqref{5} under the event-triggering update rule \eqref{4} guarantees that $x_{k}$ converges to an optimal solution of problem \eqref{1}, i.e., $\lim_{k\to \infty} x_{k}=x^{*}$, and the variables $(\lambda_{k}, s_{k})$ converge to a point $(\lambda^{*}, s^{*})$, i.e., $\lim_{k\to \infty} \lambda_{k}=\lambda^{*}$ and $\lim_{k\to \infty} s_{k}=s^{*}$.
\end{Theorem}

%

\begin{Remark}
In contrast to the distributed optimization algorithm given in \cite{9} that includes two gradient pairs and requires twice communications in each iteration, the proposed algorithm \eqref{5} only involves one gradient pair and only need once communication in each iteration. In addition, we also introduce the event-triggered mechanism and therefore our algorithm can significantly reduce communication overhead than that of \cite{9}.
\end{Remark}

\begin{Remark}
The authors of \cite{15a} recently developed a distributed algorithm to handle coupled linear equality and nonlinear inequality constraints like \eqref{1.1}, which also achieves an $O(1/k)$ convergence rate. In contrast to the results of \cite{15a}, the advantages and differences of the proposed algorithm \eqref{5} are two-fold. (i) The proposed algorithm \eqref{5} is based on the event-triggered mechanism, which reduces the communication cost and does not sacrifice the convergence rate, while the algorithm of \cite{15a} is based on periodic communication and may lead to high communication cost. (ii) A minimization optimization problem must be solved when implementing the algorithm of \cite{15a}. Conversely, our algorithm \eqref{5} is easer to implement without solving the sub-optimization problem.
\end{Remark}

Define the following function
\begin{align}\label{t1}
F(\xi)=F(x,\lambda)=\sum^{N}_{i=1}f_{i}(x_{i})+\sum^{N}_{i=1}\lambda^{T}_{i} \psi_{i}(x_{i}),
\end{align}
where $\xi=[x^{T}, \lambda^{T}]^{T}$, $x=\text{col}(x_{i})^{N}_{i=1}$ and $\lambda=\text{col}(\lambda_{i})^{N}_{i=1}$.

We give the convergence rate of algorithm \eqref{5} in the following theorem and its proof can be seen in Appendix B.

\begin{Theorem}
Under the conditions of Theorem 1, the sequence $\{\xi_{k}\}$ generated by distributed algorithm \eqref{9} satisfies that
\begin{align}\label{th1}
&\vert F(\hat{\xi}_{t})-F(\xi^{*})) \vert \le \frac{1}{2t}(\Vert \xi_{0}-\xi^{*}\Vert^{2}_{P}+\rho^{2}+2\Sigma_{t}),
\end{align}
where $\hat{\xi}_{t}=\frac{1}{t}\sum^{t-1}_{k=0}\xi_{k+1}$, $\xi^{*}=[(x^{*})^{T}, (\lambda^{*})^{T}]^{T}$, $\rho>\Vert s^{*}\Vert$, $\Sigma_{t}=(\frac{4\gamma_{2}\sqrt{N}}{\gamma_{1}}\sum^{t-1}_{k=0}E_{k}+\sqrt{\frac{2}{\gamma_{1}}}(\Vert \xi_{0}-\xi^{*}\Vert_{P}+\Vert s_{0}-s^{*}\Vert_{W})+\rho_{s}) \sum^{t-1}_{k=0}\gamma_{2} \sqrt{N}E_{k}$, $\gamma_{1}=\min(2\kappa_{c}, \lambda_{\min}(W))$, $\gamma_{2}=\max(2\beta\lambda_{\max}(L), 1)$ and $\rho_{s}=\text{sup}_{\Vert s\Vert \le \rho} \Vert s^{*}-s\Vert$.
\end{Theorem}

\begin{Remark}
Based on Lemma 1 and Proposition 1, we obtain that $F(\xi^{*})=F(x^{*}, \lambda^{*})=f(x^{*})$ since $\sum^{N}_{i=1} (\lambda^{*}_{i})^{T}$ $\psi(x^{*}_{i})=0$. Under Assumption 3 that $\sum^{\infty}_{k=0}E_{k}<\infty$, one has that $\Sigma_{t}$ is bounded for any $t\ge 0$. It then follows from \eqref{th1} that $F(\hat{\xi}_{t})$ converges to the optimal value $f(\xi^{*})$ with an $O(1/t)$ convergence rate, where $t$ is the iteration number. mechanism.
\end{Remark}

\begin{Remark}
The advantages of the proposed distributed even-triggered algorithm \eqref{5} are two-fold. (i) The proposed algorithm \eqref{5} use the constant step-sizes, and its convergence rate outperforms than those algorithms in \cite{7,8,8a,9,10} that also considered the coupled nonlinear constraints. (ii) The communication cost can be effectively reduced by using event-triggered communication mechanism.
\end{Remark}


\section{Numerical simulation}
In this section, we provide a numerical example to demonstrate the effectiveness of the distributed event-triggered algorithm \eqref{5}. Similar to the example in \cite{9}, we consider a network of $N=10$ agents for solving the optimization problem \eqref{1}, in which local objective function $f_{i}(x_{i})$ and constraint functions $g_{i}(x_{i}),h_{i}(x_{i})$ are chosen as
\begin{align*}
&f_{i}(x_{i})=a_{i}x^{2}_{i}+b_{i}x_{i}+c_{i}\text{log}(1+e^{d_{i}x_{i}}),\\
&g_{i}(x_{i})=\pi_{i}x^{2}_{i}+\varsigma_{i},\\
&h_{i}(x_{i})=\gamma_{i}x_{i}+\delta_{i},
\end{align*}
and local set constraint $\Omega_{i}=[-1,1]$ is selected for any $i\in \mathcal{V}$. The data in the functions $f_{i}, g_{i}, h_{i}$ are randomly generated from $a_{i}\in [0,2], b_{i}\in [-5,5],  c_{i}\in [0,2], d_{i}\in [0,1], \pi_{i}\in [0,2], \varsigma_{i}\in [-2,0], \gamma_{i}\in [-1,1]$ and $\delta_{i}\in [-2,2]$. Consider a ring graph for describing the network topology of ten agents. Obviously, Assumption 1 can be satisfied.

\begin{figure}[!ht]
\centering
\includegraphics[width=0.5\textwidth, clip=true]{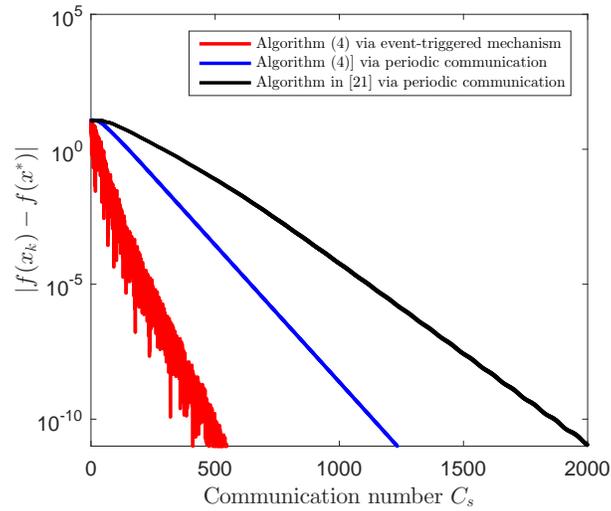}
\caption{$\vert f(x_{k})-f(x^{*})\vert$ versus the communication number $C_{s}$}
\end{figure}

We select the constant step-sizes of \eqref{5} as $\alpha=0.15$ and $\beta=1.2$. The event-triggering threshold for each agent is chosen as $\varepsilon_{i,k}=\frac{10}{e^{0.01(k+1)}}, i\in \mathcal{I}_{4}$, which satisfies the summable property in Assumption 3. The initial values of all the state variables are set as zero. Fig. 1 shows the convergence results of the objective error $\vert f(x_{k})-f(x^{*})\vert$ with respect to the communication number $C_{s}$ under algorithm \eqref{5} via event-triggered communication and periodic communication, and algorithm of \cite{9} via periodic communication. Note that $C_{s}=\frac{1}{10}\sum^{10}_{i=1}C_{i}$ denotes the average communication transmission number of ten agents and $C_{i}$ denotes agent $i$'s the communication number. From Fig. 1, we observe that the proposed algorithm \eqref{5} guarantees the objective error $\vert f(x_{k})-f(x^{*}) \vert$ converges to zero with a fast convergence speed, and the communication numbers can be significantly reduced by using the event-triggered mechanism. In addition, the specific average communication numbers $C_{s}$ at the accuracy of $10^{-10}$ for three algorithms are calculated as $n_{1}=440, n_{2}=1120$ and $n_{3}=1830$, respectively. It is shown that the proposed event-triggered algorithm \eqref{5} reduces about $75.96\%$ communication numbers than algorithm in \cite{9} at the convergence accuracy of $10^{-10}$.
%
%


\section{Conclusion}
This paper has developed a distributed event-triggered algorithm with the constant step-sizes for solving convex optimization problems with local set constraints, coupled constraints including affine equality constraint and nonlinear inequality constraint. We have shown that the proposed algorithm not only achieves an exact convergence to an optimal solution with an $O(1/k)$ convergence rate, but also can effectively reduce the communication cost. Future extensions include considering the optimization problems with coupled constraints under the general unbalanced directed graph.



%

\appendix
\renewcommand{\thesection}{Appendix A}

Before presenting the proofs of Theorems 1-2, we provide some useful Lemmas and propositions.

\begin{Lemma}[\cite{16a}]
Let $u_{k}\ge 0$ and $u_{k+1}\le u_{k}+\nu_{k}$, where $\nu_{k}\ge 0$ and $\sum^{\infty}_{k=0} \nu_{k}<\infty$. Then, $u_{k}$ is convergent as $k\to \infty$.
\end{Lemma}

\begin{Lemma}[\cite{16}]
Under Assumption 1, for any given vector $v\in \textbf{Im}(L)$, there exists a unique vector $v' \in \textbf{Im}(L)$ such that $v=Lv'$.
\end{Lemma}


\begin{Lemma}[\cite{17a}]
Suppose that $f_{i}(x_{i})$ and $\psi_{i}(x_{i})$ both are convex. It follows that (i) $F(\xi)$ defined in \eqref{t1} is a convex-concave function. (ii) $\Phi(\xi)$ shown in \eqref{10} is a monotone operator, i.e., for any $\xi_{1}, \xi_{2} \in \mathbb{R}^{r}$,
\begin{align*}
  (\Phi(\xi_{2})-\Phi(\xi_{1}))^{T}(\xi_{2}-\xi_{1})\ge 0.
\end{align*}
\end{Lemma}

\begin{Lemma}[\cite{17a}]
Define $\hat{x}_{t}=\frac{1}{t}\sum^{t-1}_{k=0}x_{k+1}$ and $\hat{\lambda}_{t}=\frac{1}{t}\sum^{t-1}_{k=0}\lambda_{k+1}$. Under Assumption 2, it follows that
\begin{align}\label{lm1}
F(\hat{x}_{t},\lambda^{*})-F(x^{*},\hat{\lambda}_{t})\le \frac{1}{t} \sum^{t-1}_{k=0} \Phi(\xi_{k+1})(\xi_{k+1}-\xi^{*})
\end{align}
\end{Lemma}

\begin{Proposition}
Under the conditions of Theorem 1, for any $\xi$ and $z=[0^{T}, s^{T}]^{T}, s\in \textbf{Im}(L\otimes I_{m})$, we have that
\begin{align}\label{t2}
&(\xi_{k+1}-\xi)^{T}(\Phi(\xi_{k+1})+z) \notag\\
&\le-\frac{\kappa_{c}}{2}\Vert \xi_{k+1}-\xi_{k}\Vert^{2}-\frac{\kappa_{c}}{2}(\Vert \xi_{k+1}-\xi_{k}\Vert^{2}\notag \\
&\quad-\Vert \xi_{k} -\xi_{k-1}\Vert^{2})-\frac{1}{2}(\Vert \xi_{k+1}-\xi\Vert^{2}_{P}-\Vert \xi_{k}-\xi\Vert^{2}_{P}) \notag\\
&\quad-\frac{1}{2}\Vert s_{k+1}-s_{k}\Vert^{2}_{W}-\frac{1}{2}(\Vert s_{k+1}-s\Vert^{2}_{W}-\Vert s_{k}-s\Vert^{2}_{W}) \notag\\
&\quad-(\xi_{k+1}-\xi)^{T}\beta L_{s}(\tilde{e}_{k}-\tilde{e}_{k+1})+\tilde{e}^{T}_{k+1}(z_{k+1}-z) \notag\\
&\quad+\xi^{T}(z_{k+1}-z)-\frac{1}{\alpha}(\xi_{k+1}-\xi)^{T}Y_{k+1}+\vartheta_{k+1}
\end{align}
where $P=\frac{1}{\alpha}I_{Nm}-\beta L_{s}$ is positive definite if $\alpha<\frac{1}{3\kappa_{c}}$ and $\beta\le \frac{1-3\alpha\kappa_{c}}{\alpha\lambda_{\max}(L)}$, $W=\frac{1}{\beta}(L+\bm 1_{N}\bm 1^{T}_{N})^{-1}\otimes I_{m}$ is positive definite, and $\vartheta_{k+1}=(\xi_{k+1}-\xi)^{T}(\Phi(\xi_{k+1})-\Phi(\xi_{k}))-(\xi_{k}-\xi)^{T}(\Phi(\xi_{k})-\Phi(\xi_{k-1}))$.
\end{Proposition}
\begin{proof}
Set $Y_{k+1}=\Gamma_{k}-\mathcal{P}_{\Lambda}(\Gamma_{k})$, where $\Gamma_{k}=\xi_{k}-2\alpha \Phi(\xi_{k})+\alpha\Phi(\xi_{k-1})-\alpha z_{i,k} -\alpha \beta L_{s}(\xi_{k}+\tilde{e}_{k})$. From \eqref{9a}, one has that
\begin{align}\label{t3}
\xi_{k+1}&=\xi_{k}-2\alpha \Phi(\xi_{k})+\alpha\Phi(\xi_{k-1})-\alpha z_{k} \notag\\
&\quad-\alpha \beta L_{s}(\xi_{k}+\tilde{e}_{k})-Y_{k+1} \notag\\
&=\xi_{k}-\alpha \Phi(\xi_{k+1})+\chi_{k+1}-\alpha z_{k} \notag\\
&\quad-\alpha \beta L_{s}(\xi_{k}+\tilde{e}_{k})-Y_{k+1},
\end{align}
where $\chi_{k+1}=\alpha\{(\Phi(\xi_{k+1})-\Phi(\xi_{k}))-(\Phi(\xi_{k})-\Phi(\xi_{k-1}))\}$.

Substituting \eqref{9b} into \eqref{t3}, with a simple merging operation, we obtain that
\begin{align*}
&\Phi(\xi_{k+1})+z_{k+1}+P(\xi_{k+1}-\xi_{k})+\beta L_{s}(\tilde{e}_{k}-\tilde{e}_{k+1})+\frac{1}{\alpha}Y_{k+1}-\frac{1}{\alpha}\chi_{k+1}=0. \notag
\end{align*}
It then follows that
\begin{align}\label{t4}
&(\xi_{k+1}-\xi)^{T}\Phi(\xi_{k+1})+(\xi_{k+1}-\xi)^{T}z\\
&=-(\xi_{k+1}-\xi)^{T}(z_{k+1}-z)-(\xi_{k+1}-\xi)^{T}P(\xi_{k+1}-\xi_{k}) \notag \\
&\quad-(\xi_{k+1}-\xi)^{T}(\beta L_{s}(\tilde{e}_{k}-\tilde{e}_{k+1})+\frac{1}{\alpha}(Y_{k+1}-\chi_{k+1})). \notag
\end{align}
From the last equation of \eqref{6} and the initial condition $s_{i,0}=0$, it follows that $s_{t+1}=\beta (L\otimes I_{m})\sum^{t+1}_{k=1}\tilde{\lambda}_{k}$ for any $t\ge 0$. This implies that $s_{k+1}\in \textbf{Im}(L\otimes I_{m})$ for $\forall k\ge 0$. Since $s_{k+1}, s \in \textbf{Im}(L\otimes I_{m})$, based on Lemma 3, there exit $s'_{k+1} $ and $s'$ such that $s_{k+1}-s=(L\otimes I_{m})(s'_{k+1}-s')$. Based on $(L\otimes I_{m})\tilde{\lambda}_{k+1}=\frac{1}{\beta}(s_{k+1}-s_{k})$ and $(s_{k+1}-s_{k})=(L\otimes I_{m})(s'_{k+1}-s'_{k})$, it follows that
\begin{align}\label{t5}
&\xi^{T}_{k+1}(z_{k+1}-z)=\lambda^{T}_{k+1}(s_{k+1}-s) \notag\\
&=\tilde{\lambda}^{T}_{k+1}(L\otimes I_{m})(s'_{k+1}-s')-e^{T}_{k+1}(s_{k+1}-s) \notag\\
&=(s'_{k+1}-s'_{k})\frac{1}{\beta}(L\otimes I_{m})(s'_{k+1}-s)-e^{T}_{k+1}(s_{k+1}-s),
\end{align}
According to \eqref{t4} and \eqref{t5}, we obtain that
\begin{align}\label{t6}
&~(\xi_{k+1}-\xi)^{T}\Phi(\xi_{k+1})+(\xi_{k+1}-\xi)^{T}z \notag \\
&=-(\xi_{k+1}-\xi)^{T}P(\xi_{k+1}-\xi_{k})-(\xi_{k+1}-\xi)^{T}\beta L_{s}(\tilde{e}_{k}\notag \\
&\quad-\tilde{e}_{k+1})-(s'_{k+1}-s'_{k})\frac{1}{\beta}\tilde{L}(s'_{k+1}-s)+e^{T}_{k+1}(s_{k+1}-s)\notag\\
&\quad+\xi^{T}(z_{k+1}-z)-\frac{1}{\alpha}(\xi_{k+1}-\xi)^{T}(Y_{k+1}-\chi_{k+1}) \notag\\
&=-\frac{1}{2}\Vert \xi_{k+1}-\xi_{k}\Vert^{2}_{P}-\frac{1}{2}(\Vert \xi_{k+1}-\xi\Vert^{2}_{P}-\Vert \xi_{k}-\xi\Vert^{2}_{P}) \notag \\
&\quad-\frac{1}{2}\Vert s'_{k+1}-s'_{k}\Vert^{2}_{\frac{1}{\beta}\tilde{L}}-\frac{1}{2}(\Vert s'_{k+1}-s'\Vert^{2}_{\frac{1}{\beta}\tilde{L}}-\Vert s'_{k}-s'\Vert^{2}_{\frac{1}{\beta}\tilde{L}})\notag \\
&\quad-(\xi_{k+1}-\xi)^{T}\beta L_{s}(\tilde{e}_{k}-\tilde{e}_{k+1})+e^{T}_{k+1}(s_{k+1}-s) \notag\\
&\quad+\xi^{T}(z_{k+1}-z)-\frac{1}{\alpha}(\xi_{k+1}-\xi)^{T}(Y_{k+1}-\chi_{k+1}),
\end{align}
where $\tilde{L}=L\otimes I_{m}$ and the last inequality is derived by using the fact that $u^{T}Mv=\frac{1}{2}\Vert u\Vert^{2}_{M}+\frac{1}{2}\Vert v \Vert^{2}_{M}-\frac{1}{2}\Vert u-v \Vert^{2}_{M}$ holds for any vector $u, v$ and positive semi-definite matrix $M$. In addition, it follows from the definition of $\chi_{k+1}$ that
\begin{align}\label{t7a}
&\frac{1}{\alpha}(\xi_{k+1}-\xi)^{T}\chi_{k+1}=(\xi_{k+1}-\xi)^{T}(\Phi(\xi_{k+1})-\Phi(\xi_{k})) \notag\\
&~~~-(\xi_{k+1}-\xi)^{T}(\Phi(\xi_{k})-\Phi(\xi_{k-1})) \notag \\
&=(\xi_{k+1}-\xi)^{T}(\Phi(\xi_{k+1})-\Phi(\xi_{k}))-(\xi_{k}-\xi)^{T}(\Phi(\xi_{k}) \notag \\
&~~~-\Phi(\xi_{k-1}))+(\xi_{k}-\xi_{k+1})^{T}(\Phi(\xi_{k})-\Phi(\xi_{k-1})) \notag \\
&\le (\xi_{k+1}-\xi)^{T}(\Phi(\xi_{k+1})-\Phi(\xi_{k}))-(\xi_{k}-\xi)^{T}(\Phi(\xi_{k}) \notag \\
&~~~-\Phi(\xi_{k-1}))+\frac{\kappa_{c}}{2}\Vert \xi_{k+1}-\xi_{k}\Vert^{2}+\frac{\kappa_{c}}{2}\Vert \xi_{k}-\xi_{k-1}\Vert^{2} \notag \\
&=\frac{\kappa_{c}}{2}\Vert \xi_{k+1}-\xi_{k}\Vert^{2}+\frac{\kappa_{c}}{2}\Vert \xi_{k}-\xi_{k-1}\Vert^{2}+\vartheta_{k+1},
\end{align}
where the last inequality is obtained due to the Lipschitz continuity of $\Phi(\xi)$ and Young's inequality.

Since $\alpha<\frac{1}{3\kappa_{c}}$ and $\beta\le \frac{1-3\alpha\kappa_{c}}{\alpha\lambda_{\max}(L)}$, we have that $P-3\kappa_{c}I_{Nm}$ is positive semi-definite, and one can further derive that
\begin{align}\label{t7f}
\frac{1}{2}\Vert \xi_{k+1}-\xi_{k}\Vert^{2}_{P}\ge \frac{3\kappa_{c}}{2}\Vert \xi_{k+1}-\xi_{k}\Vert^{2}.
\end{align}
It then follows that
\begin{align}\label{t7b}
&-\frac{1}{2}\Vert \xi_{k+1}-\xi_{k}\Vert^{2}_{P}+\frac{\kappa_{c}}{2}(\Vert \xi_{k+1}-\xi_{k}\Vert^{2}+\Vert \xi_{k}-\xi_{k-1}\Vert^{2})\\
&\le-\frac{\kappa_{c}}{2}\Vert \xi_{k+1}-\xi_{k}\Vert^{2}-\frac{\kappa_{c}}{2}(\Vert \xi_{k+1}-\xi_{k}\Vert^{2}-\Vert \xi_{k}-\xi_{k-1}\Vert^{2}). \notag
\end{align}
Based on \eqref{t7a} and \eqref{t7b}, Eq. \eqref{t6} is calculated as
\begin{align}\label{pm1}
&~(\xi_{k+1}-\xi)^{T}\Phi(\xi_{k+1})+(\xi_{k+1}-\xi)^{T}z \notag \\
&=-\frac{\kappa_{c}}{2}\Vert \xi_{k+1}-\xi_{k}\Vert^{2}-\frac{\kappa_{c}}{2}(\Vert \xi_{k+1}-\xi_{k}\Vert^{2}-\Vert \xi_{k}-\xi_{k-1}\Vert^{2}) \notag \\
&\quad-\frac{1}{2}(\Vert \xi_{k+1}-\xi\Vert^{2}_{P}-\Vert \xi_{k}-\xi\Vert^{2}_{P})-\frac{1}{2}\Vert s'_{k+1}-s'_{k}\Vert^{2}_{\frac{1}{\beta}\tilde{L}}\notag \\
&\quad-\frac{1}{2}(\Vert s'_{k+1}-s'\Vert^{2}_{\frac{1}{\beta}\tilde{L}}-\Vert s'_{k}-s'\Vert^{2}_{\frac{1}{\beta}\tilde{L}})\notag \\
&\quad-(\xi_{k+1}-\xi)^{T}\beta L_{s}(\tilde{e}_{k}-\tilde{e}_{k+1})+e^{T}_{k+1}(s_{k+1}-s) \notag\\
&\quad+\xi^{T}(z_{k+1}-z)-\frac{1}{\alpha}(\xi_{k+1}-\xi)^{T}Y_{k+1}+\vartheta_{k+1}.
\end{align}
By considering $s_{k+1}, s_{k}, s \in \textbf{Im}(L\otimes I_{m})$ and $\bm 1^{T}_{N}L=0$, one has that
\begin{align*}
s_{k+1}-s_{k}=\tilde{L}(s'_{k+1}-s'_{k})=(\tilde{L}+\bm 1_{N}\bm 1^{T}_{N}\otimes I_{m})(s'_{k+1}-s'_{k}),
\end{align*}
Note that $\tilde{L}+\bm 1_{N}\bm 1^{T}_{N}\otimes I_{m}=(L+\bm 1_{N}\bm 1^{T}_{N})\otimes I_{m}$ is positive definite, and one has that $s'_{k+1}-s'_{k}=((L+\bm 1_{N}\bm 1^{T}_{N})^{-1}\otimes I_{m})(s_{k+1}-s_{k})$. It then follows that
\begin{align}\label{t7}
&\Vert s'_{k+1}-s'_{k}\Vert^{2}_{\frac{1}{\beta}\tilde{L}}=\frac{1}{\beta}(s_{k+1}-s_{k})^{T}(s'_{k+1}-s'_{k}) \notag\\
&=(s_{k+1}-s_{k})^{T}\frac{1}{\beta}((L+\bm 1_{N}\bm 1^{T}_{N})^{-1}\otimes I_{m})(s_{k+1}-s_{k}) \notag \\
&=\Vert s_{k+1}-s_{k}\Vert^{2}_{W}.
\end{align}
Similarly, we can derive that $\Vert s'_{k+1}-s'\Vert^{2}_{\frac{1}{\beta}\tilde{L}}=\Vert s_{k+1}-s\Vert^{2}_{W}$ and $\Vert s'_{k}-s'\Vert^{2}_{\frac{1}{\beta}\tilde{L}}=\Vert s_{k}-s\Vert^{2}_{W}$. Also, one can derive that
\begin{align}\label{t7c}
e^{T}_{k+1}(s_{k+1}-s)=\tilde{e}^{T}_{k+1}(z_{k+1}-z_{k}).
\end{align}
Substituting \eqref{t7} and \eqref{t7c} into \eqref{pm1}, we obtain the desired inequality \eqref{t2}.
\end{proof}

\begin{Proposition}
Under the conditions of Theorem 1, $\forall k, t\in \mathbb{N}$, the following inequality holds for any $k\le t$
\begin{align}\label{t8}
&\Vert \xi_{k+1}-\xi^{*} \Vert+\Vert s_{k+1}-s^{*}\Vert \le \frac{4\gamma_{2}\sqrt{N}}{\gamma_{1}}\sum^{t-1}_{k=0}E_{k}\\
&~~~~~~+\sqrt{\frac{2}{\gamma_{1}}}(\Vert \xi_{0}-\xi^{*}\Vert_{P}+\Vert s_{0}-s^{*}\Vert_{W}). \notag
\end{align}
\end{Proposition}
\begin{proof}
Define $\xi^{*}=[(x^{*})^{T}, (\lambda^{*})^{T}]^{T}$ and $z^{*}=[0^{T},(s^{*})^{T}]^{T}$. From \eqref{7} and \eqref{9}, we have that $-(\Phi(\xi^{*})+z^{*}) \in \mathcal{N}_{\Lambda}(\xi^{*})$. Since $\xi_{k+1} \in \Lambda$, it follows that
\begin{align}\label{t9}
(\Phi(\xi^{*})+z^{*})^{T}(\xi_{k+1}-\xi^{*})\ge 0.
\end{align}
According to the monotone property of $\Phi(\xi)$ given in Lemma 4, one has that
\begin{align}\label{t10}
(\xi_{k+1}-\xi^{*})^{T}(\Phi(\xi_{k+1})-\Phi(\xi^{*}))\ge 0.
\end{align}
By combining \eqref{t9} and \eqref{t10}, we can derive that
\begin{align}\label{t10a}
&(\xi_{k+1}-\xi^{*})^{T}(\Phi(\xi_{k+1})+z^{*})=(\xi_{k+1}-\xi^{*})^{T}(\Phi(\xi_{k+1}) \notag\\
&\quad-\Phi(\xi^{*}))+(\xi_{k+1}-\xi^{*})^{T}(\Phi(\xi^{*})+z^{*})\ge 0.
\end{align}
Setting $\xi=\xi^{*}$ and $z=z^{*}$ of \eqref{t2}, and then summing it over $k$ from $0$ to $t-1$ yields
\begin{align}\label{t11}
&0 \le \sum^{t-1}_{k=0} (\xi_{k+1}-\xi^{*})^{T}(\Phi(\xi_{k+1})+z^{*})\notag \\
&\le-\frac{\kappa_{c}}{2}\sum^{t-1}_{k=0}\Vert \xi_{k+1}-\xi_{k}\Vert^{2}-\frac{\kappa_{c}}{2}\Vert \xi_{t}-\xi_{t-1}\Vert^{2} \notag \\
&\quad-\frac{1}{2}(\Vert \xi_{t}-\xi^{*}\Vert^{2}_{P}-\Vert \xi_{0}-\xi^{*}\Vert^{2}_{P})-\frac{1}{2}\sum^{t-1}_{k=0}\Vert s_{k+1}-s_{k}\Vert^{2}_{W} \notag\\
&\quad-\frac{1}{2}(\Vert s_{t}-s^{*}\Vert^{2}_{W}-\Vert s_{0}-s^{*}\Vert^{2}_{W}) \notag\\
&\quad+\sum^{t-1}_{k=0}\{(\xi^{*}-\xi_{k+1})^{T}\beta L_{s}(\tilde{e}_{k}-\tilde{e}_{k+1})+(z_{k+1}-z^{*})^{T}\tilde{e}_{k+1}\} \notag\\
&\quad+\sum^{t-1}_{k=0}\{(\xi^{*})^{T}(z_{k+1}-z^{*})-\frac{1}{\alpha}(\xi_{k+1}-\xi^{*})^{T}Y_{k+1}\} \notag\\
&\quad+(\xi_{t}-\xi^{*})^{T}(\Phi(\xi_{t})-\Phi(\xi_{t-1})),
\end{align}
where the first inequality is derived by \eqref{t10a}, and the second inequality is obtained by using $\sum^{t-1}_{k=0}-\frac{\kappa_{c}}{2}(\Vert \xi_{k+1}-\xi_{k}\Vert^{2}-\Vert \xi_{k} -\xi_{k-1}\Vert^{2})=-\frac{\kappa_{c}}{2}(\Vert \xi_{t}-\xi_{t-1}\Vert^{2}-\Vert \xi_{0} -\xi_{-1}\Vert^{2})=-\frac{\kappa_{c}}{2}\Vert \xi_{t}-\xi_{t-1}\Vert^{2}$ since $\xi_{0}=\xi_{-1}$, and $\sum^{t-1}_{k=0}\vartheta_{k+1}=\sum^{t-1}_{k=0}(\xi_{k+1}-\xi^{*})^{T}(\Phi(\xi_{k+1})-\Phi(\xi_{k}))-(\xi_{k}-\xi^{*})^{T}(\Phi(\xi_{k})-\Phi(\xi_{k-1}))=(\xi_{t}-\xi^{*})^{T}(\Phi(\xi_{t})-\Phi(\xi_{t-1}))$.

Note that $(\xi^{*})^{T}(z_{k+1}-z^{*})=(\lambda^{*})^{T}(s_{k+1}-s^{*})$.  Since $s_{k+1}, s^{*} \in \textbf{Im}(L\otimes I_{m})$ and $(L\otimes I_{m})\lambda^{*}=0$, we obtain that $(\xi^{*})^{T}(z_{k+1}-z^{*})=0$ for $\forall k\ge 0$. According to the definition $Y_{k+1}$ in Proposition 2, one has that $\mathcal{P}_{\Lambda}(\xi_{k+1}+Y_{k+1})=\xi_{k+1}$. This implies that $Y_{k+1}\in \mathcal{N}_{\Lambda}(\xi_{k+1})$, and one further can derive that $(\xi_{k+1}-\xi^{*})^{T}Y_{k+1} \ge 0$. The following inequality holds that for $\forall k\ge 0$
\begin{align}\label{t12y}
(\xi^{*})^{T}(z_{k+1}-z^{*})-\frac{1}{\alpha}(\xi_{k+1}-\xi^{*})^{T}Y_{k+1}\le 0.
\end{align}
In addition, since $(\xi_{t}-\xi^{*})^{T}(\Phi(\xi_{t})-\Phi(\xi_{t-1}))\le \frac{\kappa_{c}}{2}\Vert \xi_{t}-\xi^{*}\Vert^{2}+\frac{\kappa_{c}}{2} \Vert \xi_{t}-\xi_{t-1}\Vert^{2}$ and $\frac{1}{2}\Vert \xi_{t}-\xi^{*}\Vert^{2}_{P} \ge \frac{3\kappa_{c}}{2}\Vert \xi_{t}-\xi^{*}\Vert^{2}$, we obtain that
\begin{align}\label{t12a}
&-\frac{1}{2}\Vert \xi_{t}-\xi^{*}\Vert^{2}_{P}-\frac{\kappa_{c}}{2}\Vert \xi_{t}-\xi_{t-1}\Vert^{2}+(\xi_{t}-\xi^{*})^{T}(\Phi(\xi_{t}) \notag\\
&~~~~~~~-\Phi(\xi_{t-1}))\le -\kappa_{c}\Vert \xi_{t}-\xi^{*}\Vert^{2}.
\end{align}
Based on \eqref{t11}, \eqref{t12y} and \eqref{t12a}, we have that
\begin{align}\label{t12}
&\kappa_{c}\Vert \xi_{t}-\xi^{*}\Vert^{2}+\frac{1}{2}\Vert s_{t}-s^{*}\Vert^{2}_{W} \notag \\
\le &\frac{1}{2}(\Vert \xi_{0}-\xi^{*}\Vert^{2}_{P}+\Vert s_{0}-s^{*}\Vert^{2}_{W})+\sum^{t-1}_{k=0}\{(\xi^{*}-\xi_{k+1})^{T} \notag \\
&\beta L_{s}(\tilde{e}_{k}-\tilde{e}_{k+1})+(z_{k+1}-z^{*})^{T}\tilde{e}_{k+1}\}.
\end{align}
Since $\Vert e_{i,k}\Vert \le \varepsilon_{i,k}$ and $E_{k}=\max_{i\in \mathcal{I}_{N}}(\varepsilon_{i,k})$, we obtain that $\Vert \tilde{e}_{k}\Vert=\Vert e_{k}\Vert\le \sqrt{N}E_{k}$ for $\forall k\ge 0$. According to the nonincreasing of $E_{k}$ and Cauchy-Schwarz inequality, it follows from \eqref{t12} that
\begin{align}\label{t13}
&\frac{1}{4}\gamma_{1}(\Vert \xi_{t}-\xi^{*} \Vert+\Vert s_{t}-s^{*}\Vert)^{2} \notag\\
&\le \kappa_{c}\Vert \xi_{t}-\xi^{*}\Vert^{2}+\frac{1}{2}\Vert s_{t}-s^{*}\Vert^{2}_{W} \notag\\
&\le \frac{1}{2}(\Vert \xi_{0}-\xi^{*}\Vert^{2}_{P}+\Vert s_{0}-s^{*}\Vert^{2}_{W})  \\
&~~~~+\sum^{t-1}_{k=0}\gamma_{2} \sqrt{N}E_{k}(\Vert \xi_{k+1}-\xi^{*}\Vert+\Vert s_{k+1}-s^{*} \Vert), \notag
\end{align}
where $\gamma_{1}=\min(2\kappa_{c}, \lambda_{\min}(W))$ and $\gamma_{2}=\max(2\beta\lambda_{\max}(L),$ $1)$. By using Lemma 1 in \cite{17}, it follows from \eqref{t13} that
\begin{align*}
&\Vert \xi_{t}-\xi^{*} \Vert+\Vert s_{t}-s^{*}\Vert\\
&\le \frac{2\gamma_{2}\sqrt{N}}{\gamma_{1}}\sum^{t-1}_{k=0}E_{k}+\{\frac{2}{\gamma_{1}}(\Vert \xi_{0}-\xi^{*}\Vert^{2}_{P}+\Vert s_{0}-s^{*}\Vert^{2}_{W})\\
&\quad +(\frac{2\gamma_{2}\sqrt{N}}{\gamma_{1}}\sum^{t-1}_{k=0}E_{k})^{2}\}^{1/2}\\
&\le \frac{4\gamma_{2}\sqrt{N}}{\gamma_{1}}\sum^{t-1}_{k=0}E_{k}+\sqrt{\frac{2}{\gamma_{1}}}(\Vert \xi_{0}-\xi^{*}\Vert_{P}+\Vert s_{0}-s^{*}\Vert_{W}),
\end{align*}
where the last inequality is derived by using $\sqrt{\Vert a\Vert^{2}+\Vert b\Vert^{2}}\le \Vert a\Vert+\Vert b\Vert$ for $\forall a,b$. Due to the monotone increasing of $\sum^{t-1}_{k=0}E_{k}$, we obtain the inequality \eqref{t8}.
\end{proof}
Based on the results of Lemmas 2-5 and Propositions 2-3, we next give the proofs of Theorems 1 and 2.
\vspace{-0.2cm}
\subsection{Proof of Theorem 1}
According to \eqref{t11} and \eqref{t12a}, we obtain that
\begin{align*}
&0\le -\frac{\kappa_{c}}{2}\sum^{t-1}_{k=0}\Vert \xi_{k+1}-\xi_{k}\Vert^{2}-\frac{\kappa_{c}}{2}\Vert \xi_{t}-\xi^{*}\Vert^{2} \\
&+\frac{1}{2}\Vert \xi_{0}-\xi^{*}\Vert^{2}_{P}-\frac{1}{2}\sum^{t-1}_{k=0}\Vert s_{k+1}-s_{k}\Vert^{2}_{W} \notag\\
&-\frac{1}{2}(\Vert s_{t}-s^{*}\Vert^{2}_{W}-\Vert s_{0}-s^{*}\Vert^{2}_{W})+\sum^{t-1}_{k=0}\{(\xi^{*}-\xi_{k+1})^{T} \notag \\
&\beta L_{s}(\tilde{e}_{k}-\tilde{e}_{k+1})+(z_{k+1}-z^{*})^{T}\tilde{e}_{k+1}\}.
\end{align*}
It then follows that
\begin{align*}
&\frac{1}{2}\sum^{t-1}_{k=0}(\kappa_{c}\Vert \xi_{k+1}-\xi_{k}\Vert^{2}+\Vert s_{k+1}-s_{k}\Vert^{2}_{W})\\
&\le  \frac{1}{2}(\Vert \xi_{0}-\xi^{*}\Vert^{2}_{P}+\Vert s_{0}-s^{*}\Vert^{2}_{W})\\
&~~~~+\sum^{t-1}_{k=0}\gamma_{2} \sqrt{N}E_{k}(\Vert \xi_{k+1}-\xi^{*}\Vert+\Vert s_{k+1}-s^{*} \Vert).  \notag
\end{align*}
Based on Proposition 3, we obtain that
\begin{align*}
&\sum^{t-1}_{k=0}\gamma_{2} \sqrt{N}E_{k}(\Vert \xi_{k+1}-\xi^{*}\Vert+\Vert s_{k+1}-s^{*} \Vert)\\
&~~\le (\frac{4\gamma_{2}\sqrt{N}}{\gamma_{1}}\sum^{t-1}_{k=0}E_{k}+\sqrt{\frac{2}{\gamma_{1}}}(\Vert \xi_{0}-\xi^{*}\Vert_{P}\\
&~~+\Vert s_{0}-s^{*}\Vert_{W})) \sum^{t-1}_{k=0}\gamma_{2} \sqrt{N}E_{k}.
\end{align*}
Under Assumption 3 that $\sum^{\infty}_{k=0}E_{k}<\infty$, we obtain that $\sum^{t-1}_{k=0}\gamma_{2} \sqrt{N}E_{k}(\Vert \xi_{k+1}-\xi^{*}\Vert+\Vert s_{k+1}-s^{*} \Vert)< \infty$ for any $\forall t\ge 0$. Then, by letting $t \to \infty$, it follows that
\begin{align*}
\sum^{\infty}_{k=0}(\kappa_{c}\Vert \xi_{k+1}-\xi_{k}\Vert^{2}+\Vert s_{k+1}-s_{k}\Vert^{2}_{W})< \infty.
\end{align*}
%
Thus, we derive that $\lim_{k \to \infty}(\xi_{k+1}-\xi_{k})=0$ and $\lim_{k \to \infty}(s_{k+1}-s_{k})=0$. According to \eqref{t8}, one has that $(\xi_{k}, s_{k})$ is bounded for $\forall k\in \mathbb{N}$, and then we obtain that $\{\xi_{k}\}$ and $\{s_{k}\}$ contain the convergent subsequences $\{\xi_{n_{k}}\}$ and $\{s_{n_{k}}\}$ to some limit points $\xi^{0}$ and $s^{0}$, i.e., $\lim_{k\to \infty} \xi_{n_{k}}=\xi^{0}=[(x^{0})^{T}, (\lambda^{0})^{T}]^{T}$ and $\lim_{k\to \infty} s_{n_{k}}=s^{0}$. Since $\lim_{k\to\infty} E_{k}=0$, it implies that $\lim_{k\to \infty}\tilde{e}_{k}=0$. From \eqref{9}, it follows that
\begin{align*}
&\xi^{0}=\mathcal{P}_{\Lambda}(\xi^{0}-\alpha \Phi(\xi^{0})-\alpha z^{0}-\alpha \beta L_{s}\xi^{0}),\\
&L_{s}\xi^{0}=0,
\end{align*}
where $z^{0}=[0^{T}, (s^{0})^{T}]^{T}$. This implies that $(\xi^{0}, z^{0})$ is a fixed point of \eqref{9}. According to Proposition 1, we know that $x^{0}$ is an optimal solution of problem \eqref{1}.

We next prove the primal sequences $\{\xi_{k}\}$ and $\{s_{k}\}$ are convergent. Setting $\xi=\xi^{*}$ and $z=z^{*}$ of \eqref{t2}, one has that
\begin{align}
&0\le (\xi_{k+1}-\xi^{*})^{T}(\Phi(\xi_{k+1})+z^{*}) \notag\\
&\le -\frac{1}{2}(\Vert \xi_{k+1}-\xi^{*}\Vert^{2}_{P}-\Vert \xi_{k}-\xi^{*}\Vert^{2}_{P})-\frac{1}{2}(\Vert s_{k+1}-s^{*}\Vert^{2}_{W} \notag \\
&\quad-\Vert s_{k}-s^{*}\Vert^{2}_{W})-\frac{\kappa_{c}}{2}(\Vert \xi_{k+1} -\xi_{k}\Vert^{2}-\Vert \xi_{k} -\xi_{k-1}\Vert^{2}) \notag \\
&\quad+\gamma_{2} \sqrt{N}E_{k}(\Vert \xi_{k+1}-\xi^{*}\Vert+\Vert s_{k+1}-s^{*}\Vert) \notag \\
&\quad+(\xi_{k+1}-\xi^{*})^{T}(\Phi(\xi_{k+1})-\Phi(\xi_{k}))\notag \\
&\quad-(\xi_{k}-\xi^{*})^{T}(\Phi(\xi_{k})-\Phi(\xi_{k-1})). \notag
\end{align}
It then follows that
\begin{align}
&\frac{1}{2}\Vert \xi_{k+1}-\xi^{*}\Vert^{2}_{P}+\frac{1}{2} \Vert s_{k+1}-s^{*}\Vert^{2}_{W} \notag \\
&\quad-(\xi_{k+1}-\xi^{*})^{T}(\Phi(\xi_{k+1})-\Phi(\xi_{k}))+\frac{\kappa_{c}}{2}\Vert \xi_{k+1} -\xi_{k}\Vert^{2} \notag \\
&\le \frac{1}{2}\Vert \xi_{k}-\xi^{*}\Vert^{2}_{P}+\frac{1}{2}\Vert s_{k}-s^{*}\Vert^{2}_{W}   \notag \\
&\quad-(\xi_{k}-\xi^{*})^{T}(\Phi(\xi_{k})-\Phi(\xi_{k-1}))+\frac{\kappa_{c}}{2}\Vert \xi_{k} -\xi_{k}\Vert^{2} \notag \\
&\quad+\gamma_{2} \sqrt{N}E_{k}(\Vert \xi_{k+1}-\xi^{*}\Vert+\Vert s_{k+1}-s^{*}\Vert). \label{thm1}
\end{align}
Define the variable $\Delta_{k+1}=\frac{1}{2}\Vert \xi_{k+1}-\xi^{*}\Vert^{2}_{P}+\frac{1}{2}\Vert s_{k+1}-s^{*}\Vert^{2}_{W}-(\xi_{k+1}-\xi^{*})^{T}(\Phi(\xi_{k+1})-\Phi(\xi_{k}))+\frac{\kappa_{c}}{2}\Vert \xi_{k+1}-\xi_{k}\Vert^{2}$ and $\varrho_{k}=\gamma_{2} \sqrt{N}E_{k}(\Vert \xi_{k+1}-\xi^{*}\Vert+\Vert s_{k+1}-s^{*}\Vert)$. The above equation \eqref{thm1} is simplified as
\begin{align}\label{thm2}
\Delta_{k+1}\le \Delta_{k}+\varrho_{k}.
\end{align}
By using \eqref{t7f} and Young inequality, one has that $\Delta_{k+1}\ge \kappa_{c}\Vert \xi_{k+1}-\xi^{*}\Vert^{2}+\frac{1}{2}\Vert s_{k+1}-s^{*}\Vert^{2}_{W} \ge 0$ for $\forall k\ge0$. In addition, from Assumption 3 and the boundedness of $\Vert \xi_{k+1}-\xi^{*}\Vert+\Vert s_{k+1}-s^{*}\Vert$, we have that $\varrho_{k}\ge 0$ and $\sum^{\infty}_{k=1} \varrho_{k}<\infty$. According to Lemma 2, we obtain that $\Delta_{k+1}$ is convergent. Since $\lim_{k \to \infty}(\xi_{k+1}-\xi_{k})=0$ and the Lipshitz condition $\Vert \Phi(\xi_{k+1})-\Phi(\xi_{k})\Vert\le \kappa_{c} \Vert \xi_{k+1}-\xi_{k}\Vert$, it implies that $\frac{1}{2}\Vert \xi_{k+1}-\xi^{*}\Vert^{2}_{P}+\frac{1}{2}\Vert s_{k+1}-s^{*}\Vert^{2}_{W}$ is convergent. By setting $(\xi^{*}, z^{*})=(\xi^{0}, z^{0})$ and using $\lim_{k\to \infty} \xi_{n_{k}}=\xi^{0}$ and $\lim_{k\to \infty} s_{n_{k}}=s^{0}$, it follows that
\begin{align}
\lim_{k\to \infty} \frac{1}{2}\Vert \xi_{k+1}-\xi^{*}\Vert^{2}_{P}+\frac{1}{2}\Vert s_{k+1}-s^{*}\Vert^{2}_{W}=0.
\end{align}
Thus, we obtain that $\lim_{k\to \infty} x_{k}=x^{*}$, $\lim_{k\to \infty}\lambda_{k}=\lambda^{*}$ and $\lim_{k\to\infty} s_{k}=s^{*}$.

\vspace{-0.2cm}
\subsection{Proof of Theorem 2}
Setting $\xi=\xi^{*}$ of \eqref{t2} and summing it over $k$ from $0$ to $t-1$, for any $z=[0^{T}, s^{T}]^{T}, s\in \textbf{Im}(L\otimes I_{m})$, it follows that
\begin{align}\label{t14}
& \sum^{t-1}_{k=0}(\xi_{k+1}-\xi^{*})^{T}\Phi(\xi_{k+1})+(\xi_{k+1}-\xi^{*})^{T}z  \notag \\
&\le  \frac{1}{2}(\Vert \xi_{0}-\xi^{*}\Vert^{2}_{P}+\Vert s_{0}-s\Vert^{2})\\
&~~~~+\sum^{t-1}_{k=0}\gamma_{2} \sqrt{N}E_{k}(\Vert \xi_{k+1}-\xi^{*}\Vert+\Vert s_{k+1}-s \Vert), \notag
\end{align}
where the above inequality can be derived from \eqref{t11}, \eqref{t12y} and \eqref{t12a}. In addition, according to Lemma 2.2 in \cite{18}, the following inequality holds for any $\rho >0$
\begin{align}\label{t16}
&\sum^{t-1}_{k=0}(\xi_{k+1}-\xi^{*})^{T}\Phi(\xi_{k+1}) \notag \\
&\le \sum^{t-1}_{k=0}\{(\xi_{k+1}-\xi^{*})^{T}\Phi(\xi_{k+1})+\rho\Vert \lambda_{k+1}-\lambda^{*}\Vert \} \notag\\
& \le \text{sup}_{\Vert s\Vert \le \rho} \{\frac{1}{2}(\Vert \xi_{0}-\xi^{*}\Vert^{2}_{P}+\Vert s\Vert^{2})  \notag\\
&\quad+\sum^{t-1}_{k=0}\gamma_{2} \sqrt{N}E_{k}(\Vert \xi_{k+1}-\xi^{*}\Vert+\Vert s_{k+1}-s \Vert)\},
\end{align}
where $s_{0}=0$ for initialization is used to obtain the last inequality. Based on $\Vert s_{k+1}-s\Vert \le \Vert s_{k+1}-s^{*}\Vert+\Vert s^{*}-s\Vert$ and by using \eqref{t8} of Proposition 3, one has that
\begin{align} \label{t17}
&\sum^{t-1}_{k=0}\gamma_{2} \sqrt{N}E_{k}(\Vert \xi_{k+1}-\xi^{*}\Vert+\Vert s_{k+1}-s \Vert) \\
&\le (\frac{4\gamma_{2}\sqrt{N}}{\gamma_{1}}\sum^{t-1}_{k=0}E_{k}+\sqrt{\frac{2}{\gamma_{1}}}(\Vert \xi_{0}-\xi^{*}\Vert_{P} \notag \\
&~~+\Vert s_{0}-s^{*}\Vert_{W})) \sum^{t-1}_{k=0}\gamma_{2} \sqrt{N}E_{k}+ \Vert s^{*}-s \Vert\sum^{t-1}_{k=0}\gamma_{2} \sqrt{N}E_{k}. \notag
\end{align}
Substituting \eqref{t17} into \eqref{t16} yields
\begin{align*}
\sum^{t-1}_{k=0}(\xi_{k+1}-\xi^{*})^{T}\Phi(\xi_{k+1})\le \frac{1}{2}(\Vert \xi_{0}-\xi^{*}\Vert^{2}_{P}+\rho^{2}+2\Sigma_{t}).
\end{align*}
where $\Sigma_{t}$ is defined in Theorem 2.

By using Lemma 5, it follows that
\begin{align}
&F(\hat{x}_{t},\lambda^{*})-F(x^{*},\hat{\lambda}_{t})\le \frac{1}{t}\sum^{t-1}_{k=0}(\xi_{k+1}-\xi^{*})^{T}\Phi(\xi_{k+1}) \notag\\
& \le \frac{1}{2t}(\Vert \xi_{0}-\xi^{*}\Vert^{2}_{P}+\rho^{2}+2\Sigma_{t}).
\end{align}
Note that $F(\hat{x}_{t},\lambda^{*})-F(x^{*},\hat{\lambda}_{t})=F(\hat{x}_{t},\lambda^{*})-F(x^{*},\lambda^{*})+ F(x^{*}, \lambda^{*})-F(x^{*},\hat{\lambda}_{t})$. According to the convexity and concavity properties of $F(x,\lambda)$, we easily obtain that
\begin{align*}
&0 \le F(\hat{x}_{t},\lambda^{*})-F(x^{*},\lambda^{*})\le  \frac{1}{2t}(\Vert \xi_{0}-\xi^{*}\Vert^{2}_{P}+\rho^{2}+2\Sigma_{t}),\\
&0 \le F(x^{*}, \lambda^{*})-F(x^{*},\hat{\lambda}_{t})\le  \frac{1}{2t}(\Vert \xi_{0}-\xi^{*}\Vert^{2}_{P}+\rho^{2}+2\Sigma_{t}).
\end{align*}
In addition, since $F(\hat{x}_{t},\hat{\lambda}_{t})\le F(\hat{x}_{t},\lambda^{*})$ and $F(\hat{x}_{t},\hat{\lambda}_{t})\ge F(x^{*},\hat{\lambda}_{t})$, we further derive that
\begin{align*}
&F(\hat{x}_{t},\hat{\lambda}_{t})-F(x^{*},\lambda^{*}) \le F(\hat{x}_{t},\lambda^{*})-F(x^{*},\lambda^{*}) \\
&~~~~~~~~\le  \frac{1}{2t}(\Vert \xi_{0}-\xi^{*}\Vert^{2}_{P}+\rho^{2}+2\Sigma_{t}),\\
&F(x^{*}, \lambda^{*})-F(\hat{x}_{t},\hat{\lambda}_{t})\le F(x^{*}, \lambda^{*})-F(x^{*},\hat{\lambda}_{t})\\
&~~~~~~~~\le  \frac{1}{2t}(\Vert \xi_{0}-\xi^{*}\Vert^{2}_{P}+\rho^{2}+2\Sigma_{t}).
\end{align*}
This implies that
\begin{align*}
\vert F(\hat{x}_{t},\hat{\lambda}_{t})-F(x^{*},\lambda^{*}) \vert \le \frac{1}{2t}(\Vert \xi_{0}-\xi^{*}\Vert^{2}_{P}+\rho^{2}+2\Sigma_{t}).
\end{align*}
Thus, the desired inequality \eqref{th1} is obtained.

\bibliographystyle{elsart-harv}
\bibliography{refs}

%

%
%
%




\end{document}